\documentclass[12pt]{amsart}

\usepackage[dvips]{graphicx}
	\usepackage{amsmath,amssymb,amsthm,wrapfig,amsfonts,enumerate,latexsym}

	\newtheorem{dfn}{Definition}[section]
	\newtheorem{thm}[dfn]{Theorem}
	\newtheorem{prop}[dfn]{Proposition}
	\newtheorem{cor}[dfn]{Corollary}
	\newtheorem{lem}[dfn]{Lemma}
	\newtheorem{rem}[dfn]{Remark}
	
 	\newtheorem{claim}[dfn]{Claim}

	\newtheorem{ack}{Acknowledgements\!\!}

	\newcounter{yon}
	\numberwithin{equation}{section}

	\newcommand{\dist}{\mathop{\mathit{d}} \nolimits}
	
	\newcommand{\diam}{\mathop{\mathrm{diam}} \nolimits}

	\newcommand{\pr}{\mathop{\mathrm{pr}} \nolimits}

	\newcommand{\sep}{\mathop{\mathrm{Sep}} \nolimits}
	
	\newcommand{\obs}{\mathop{\mathrm{ObsDiam}}  \nolimits}

\begin{document}

	\title[Concentration of $1$-Lipschitz maps into $(B_{\ell^p}^{\infty},\dist_{\ell^q})$]
    {Concentration of $1$-Lipschitz maps into an infinite dimensional $\ell^p$-ball with the
    $\ell^q$-distance function}
	\author[Kei Funano]{Kei Funano}
	\address{Mathematical Institute, Tohoku University, Sendai 980-8578, JAPAN}
	\email{sa4m23@math.tohoku.ac.jp}
	\subjclass[2000]{53C21, 53C23}
	\keywords{mm-space, infinite dimensional $\ell^p$-ball, concentration of $1$-Lipschitz
    maps, L\'{e}vy group}
	\thanks{This work was partially supported by Research Fellowships of
	the Japan Society for the Promotion of Science for Young Scientists.}
	\dedicatory{}
	\date{\today}

	\maketitle


\begin{abstract}In this paper, we study the L\'{e}vy-Milman
 concentration phenomenon of $1$-Lipschitz maps into infinite
 dimensional metric spaces. Our main theorem asserts that the
 concentration to an infinite dimensional $\ell^p$-ball with the
 $\ell^q$-distance function for $1\leq p<q\leq +\infty$ is equivalent to
 the concentration to the real line.
 \end{abstract}
	\setlength{\baselineskip}{5mm}

    \section{Introduction}

    This paper is devoted to investigating the L\'{e}vy-Milman
    concentration phenomenon of 1-Lipschitz maps from mm-spaces (metric
    measure spaces) to infinite dimensional metric spaces. Here, an \emph{mm-space} is a triple
    $(X,\dist_X,\mu_X)$, where $\dist_X$ is a complete separable metric
    on a set $X$ and $\mu_X$ a finite Borel measure on $(X,\dist_X)$. The theory of concentration of $1$-Lipschitz functions was
    first introduced by V. D. Milman in his investigation of asymptotic
    geometric analysis (\cite{mil1}, \cite{mil2},
    \cite{mil3}). Nowadays, the theory blend with various areas of
    mathematics, such as geometry, functional analysis and
    infinite dimensional integration, discrete mathematics and
    complexity theory, probability theory, and so on (see \cite{ledoux},
    \cite{mil5}, \cite{milsch}, \cite{pestov2} and the references
    therein for further information).

    The theory
    of concentration of maps into general metric spaces was first
    studied by M. Gromov (\cite{gromovcat}, \cite{gromov2},
    \cite{gromov}). He established the theory by introducing the
    observable diameter $\obs_{Y}(X;-\kappa)$ for an mm-space $X$, a
    metric space $Y$, and $\kappa >0$ in \cite{gromov} (see Section 2 for the
    definition of the observable diameter). Given a sequence
	$\{ X_n \}_{n=1}^{\infty}$ of mm-spaces and a metric space $Y$, we
    note that $\lim_{n\to \infty}\obs_Y(X_n;-\kappa)= 0$ for any $\kappa >0$ if and only if for any
	sequence $\{ f_n :X_n \to Y\}_{n=1}^{\infty}$ of $1$-Lipschitz maps,
    there exists a sequence $\{ m_{f_n}\}_{n=1}^{\infty}$ of points in
    $Y$ such that
    \begin{align*}
     \lim_{n\to \infty}\mu_{X_n}(\{ x_n \in X_n \mid \dist_Y(f_n(x_n),m_{f_n})\geq
     \varepsilon  \})=0
     \end{align*}for any $\varepsilon >0$. If $\lim_{n\to
     \infty}\obs_{\mathbb{R}}(X_n;-\kappa)= 0$ for any $\kappa >0$, then
     the sequence
     $\{X_n\}_{n=1}^{\infty}$ of mm-spaces is called a \emph{L\'{e}vy
     family}. The L\'{e}vy families were first introduced and analyzed by Gromov and
     Milman in \cite{milgro}. In our previous works \cite{funano2},
     \cite{funagr}, \cite{funad}, \cite{funano1}, the author proved that
     if a metric space $Y$ is either an $\mathbb{R}$-tree, a doubling space, a
     metric graph, or a Hadamard manifold, then $\lim_{n\to \infty}
     \obs_Y(X_n;-\kappa)=0$ holds for any $\kappa >0$ and any L\'{e}vy
     family $\{X_n\}_{n=1}^{\infty}$. To prove these results, we needed
     to assume the finiteness of the dimension of the target metric spaces. 

     In this paper, we treat the case where the dimension of the target
     metric space $Y$ is
     infinite. The author has proved in \cite{funath} that if the target space $Y$ is so big that
     an mm-space $X$ with some homogeneity property can
     isometrically be embedded into $Y$, then its
     observable diameter $\obs_Y(X;-\kappa)$ is not close to zero. It
     seems from this result that the concentration to an infinite
     dimensional metric space cannot happen easily.

    A main theorem of this paper is the following. For $1\leq p\leq
    +\infty$, we denote by $B_{\ell^p}^{\infty}$ an infinite
    dimensional $\ell^p$-ball $\{(x_n)_{n=1}^{\infty}\in
    \mathbb{R}^{\infty} \mid \sum_{n=1}^{\infty}|x_n|^p \leq 1\}$ and by
    $\dist_{\ell^p}$ the $\ell^p$-distance function.
    \begin{thm}\label{mth1}Let $\{X_n\}_{n=1}^{\infty}$ be a sequence of
     mm-spaces and $1\leq p<q \leq +\infty$. Then, the sequence $\{
     X_n\}_{n=1}^{\infty}$ is a L\'{e}vy family if and only if
     \begin{align}\label{smth1}
      \lim_{n\to
      \infty}\obs_{(B_{\ell^p}^{\infty},\dist_{\ell^q})}(X_n;-\kappa)= 0
      \text{ for any }\kappa >0.
      \end{align}
     \end{thm}

     As a result, we obtain the example of the infinite dimensional
     target metric space such that the
     concentration to the space happens as often as the concentration to
     the real line.

     The proof of the sufficiency of Theorem \ref{mth1} is
     easy. A. Gournay and M. Tsukamoto's observations play important roles for the proof of the converse
     (\cite{gournay}, \cite{tsukamoto}). Answering a question of Gromov
     in \cite[Section 1.1.4]{grotop},
     Tsukamoto proved in \cite{tsukamoto} that the
     ``macroscopic'' dimension of the space
     $(B_{\ell^p}^{\infty}, \dist_{\ell^q})$ for $1\leq p<q\leq
     +\infty$ is finite. Gournay independently proved it in \cite{gournay} in the case of $q=+\infty$. For any $p$ and $q$ with $1\leq q\leq p\leq
     +\infty$, we have an
     example of a L\'{e}vy family which does
     not satisfy (\ref{smth1}) (see Proposition \ref{p4.1}).

     As applications of Theorem \ref{mth1}, by virtue of \cite[Propositions 4.3 and
     4.4]{funagr}, we obtain the following corollaries of a L\'{e}vy
     group action. A L\'{e}vy group was first introduced by Gromov and
     Milman in \cite{milgro}. Let a topological group $G$ acts on a metric space $X$. The action
    is called \emph{bounded} if for any $\varepsilon >0$ there exists a
    neighborhood $U$ of the identity element $e_{G}\in G$ such that
    $\dist_X(x,gx)<\varepsilon$ for any $g\in U$ and $x\in X$. Note that
 every bounded action is continuous. We say
    that the
    topological group $G$ acts on $X$ \emph{by uniform isomorphisms} if
    for each $g\in G$, the map
    $X\ni x\mapsto gx\in X$ is uniform continuous. The
    action is said to be \emph{uniformly equicontinuous} if for any
    $\varepsilon > 0$ there exists $\delta>0$ such that $\dist_X(gx,gy)<
    \varepsilon$ for every $g\in G$ and $x,y\in X$ with $\dist_X(x,y)<
 \delta$. Given a subset $S\subseteq G$ and $x\in X$, we put $S x:=\{gx
 \mid g\in S\}$.
     \begin{cor}Let $1\leq p < q \leq +\infty$ and assume that a
      L\'{e}vy group $G$ boundedly acts on the metric space $(B^{\infty}_{\ell^p},\dist_{\ell^q})$ by uniform isomorphisms. Then for
      any compact subset $K\subseteq G$ and any $\varepsilon >0$, there
      exists a point $x_{\varepsilon, K} \in B_{\ell^p}^{\infty}$ such
      that $\diam (K x_{\varepsilon, K})\leq \varepsilon$.
      \end{cor}

     \begin{cor}There are no non-trivial bounded uniformly
      equicontinuous actions of a L\'{e}vy group to the metric space
      $(B^{\infty}_{\ell^{p}}, \dist_{\ell^q})$ for $1\leq p< q\leq +\infty$.
      \end{cor}

Gromov and Milman pointed out in \cite{milgro} that the unitary group $U(\ell^2)$ of the separable
     Hilbert space $\ell^2$ with the strong topology is a L\'{e}vy
    group. Many concrete examples of L\'{e}vy groups are known by the
    works of S. Glasner \cite{gla}, H. Furstenberg and B. Weiss
    (unpublished), T. Giordano and V. Pestov \cite{giopes1}, \cite{giopes2}, and
    Pestov \cite{pestov1}, \cite{pestov3}. For examples, groups of measurable maps from the
    standard Lebesgue measure space to compact groups, unitary groups of some
    von Neumann algebras, groups of measure and measure-class preserving
    automorphisms of the standard Lebesgue measure space, full groups of
    amenable equivalence relations, and the isometry groups of the
    universal Urysohn metric spaces are L\'{e}vy groups (see the recent
    monograph \cite{pestov2} for precise). 

     \section{Preliminaries}

Let $Y$ be a metric space and $\nu$ a Borel measure on $Y$ such that $m:=\nu(Y)<+\infty$. 
	 We define for any $\kappa >0$
	 \begin{align*}
	  \diam (\nu , m-\kappa):= \inf \{ \diam Y_0 \mid Y_0 \subseteq Y \text{ is a Borel subset such that }\nu(Y_0)\geq m-\kappa\}
	  \end{align*}and call it the \emph{partial diameter} of $\nu$.

     \begin{dfn}[Observable diameter]\upshape Let $(X,\dist_X,\mu_X)$ be
      an mm-space with $m_X:=\mu_X(X)$ and $Y$ a metric space. For any $\kappa >0$ we
	 define the \emph{observable diameter} of $X$ by 
	 \begin{align*}
	  \obs_Y (X; -\kappa):=
	   \sup \{ \diam (f_{\ast}(\mu_X),m_X-\kappa) \mid f:X\to Y \text{ is a
      }1 \text{{\rm -Lipschitz map}}  \}, 
      \end{align*}where $f_{\ast}(\mu_X)$ stands for the push-forward
      measure of $\mu_X$ by $f$.
      \end{dfn}
      The idea of the observable diameter comes from the quantum and statistical
	mechanics, that is, we think of $\mu_X$ as a state on a configuration
	space $X$ and $f$ is interpreted as an observable.

     Let $(X,\dist_X,\mu_X)$ be an
      mm-space. For any $\kappa_1,  \kappa_2\geq 0$, we
      define the \emph{separation distance} $\sep
      (X;\kappa_1,\kappa_2)= \sep (\mu_X;\kappa_1,\kappa_2)$ of $X$ as the supremum of
      the distance $\dist_X(A,B):=\inf \{ \dist_X(a,b) \mid a\in A
      \text{ and }b\in B\}$, where $A$ and $B$ are Borel subsets
      of $X$ satisfying that $\mu_X(A)\geq \kappa_1$ and $\mu_X(B)\geq \kappa_2$.

      \begin{lem}[cf.~{\cite[Section $3\frac{1}{2}.33$]{gromov}}]\label{l2.1}
	 Let $X$ and $Y $ be two mm-spaces and $\alpha >0$. Assume that an $\alpha$-Lipschitz map $f : X\to Y$
	 satisfies $f_{\ast}(\mu_X)= \mu_Y$. Then we have 
	 \begin{align*}
	  \sep (Y; \kappa_1, \kappa_2) \leq \alpha \sep (X ; \kappa_1,\kappa_2).
	 \end{align*}
	\end{lem}

     Relationships between the observable diameter and the separation distance
     are followings. We refer to \cite[Subsection 2.2]{funad} for precise proofs.
     	\begin{lem}[{cf.~\cite[Section $3\frac{1}{2}.33$]{gromov}}]\label{l2.2}Let $X$ be an mm-space and $\kappa,\kappa' >0$ with $\kappa > \kappa'$. Then we have
	 \begin{align*}
	  \obs_{\mathbb{R}} (X ;-\kappa')\geq \sep (X;\kappa,\kappa).
	  \end{align*}
         \end{lem}

         	\begin{rem}\upshape In {\cite[Section $3\frac{1}{2}.33$]{gromov}}, Lemma \ref{l2.2} is stated as $\kappa =\kappa'$, but that is not true in general. For example, let
	 $X:=\{ x_1 , x_2\}$, $\dist_X (x_1,x_2):=1$, and $\mu_X (\{ x_1\})=\mu_X (\{  x_2   \}):= 1/2$. Putting $\kappa =\kappa'=1/2$, we have
	 $\obs_{\mathbb{R}} (X;-1/2)=0$ and $\sep
	 (X;1/2,1/2)=1$.
	 \end{rem}

   \begin{lem}[cf.~{\cite[Section
    $3\frac{1}{2}.33$]{gromov}}]\label{l2.3}Let $\nu$ be a Borel measure on
    $\mathbb{R}$ with $m:=\nu(\mathbb{R})<+\infty$. Then, for any $\kappa >0$ we have
	 \begin{align*}
	  \diam (\nu, m-2\kappa)\leq \sep (\nu; \kappa, \kappa).
      \end{align*}
      In particular, for any $\kappa >0$ we have
      \begin{align*}
        \obs_{\mathbb{R}}(X;-2\kappa)\leq \sep (X; \kappa, \kappa).
       \end{align*}
    \end{lem}

Combining Lemma \ref{l2.2} with Lemma \ref{l2.3}, we obtain the
following corollary:
    \begin{cor}[{cf.~\cite[Section $3\frac{1}{2}.33$]{gromov}}]\label{c2.1.1}A sequence $\{ X_n\}_{n=1}^{\infty}$ of mm-spaces is a
     L\'{e}vy family if and only if $\lim_{n\to \infty}\sep
     (X_n;\kappa,\kappa) =0$ for any $\kappa >0$. 
     \end{cor}

\begin{lem}\label{l2.4}Let $\nu$ be a finite Borel measure on
 $(\mathbb{R}^k,\dist_{\ell^p})$ with $m:=\nu(\mathbb{R}^k)$. Then for
 any $\kappa>0$ we have
 \begin{align*}
  \diam (\nu,m-\kappa)\leq k^{1/p}\sep \Big(\nu;\frac{\kappa}{2k},\frac{\kappa}{2k}\Big).
  \end{align*}
 \begin{proof}For $i=1,2, \cdots , k$, let $\pr_i:\mathbb{R}^k \ni (x_i)_{i=1}^k \mapsto x_i \in
  \mathbb{R}$ be the projection. For Borel subsets $A_1, A_2, \cdots , A_k \subseteq
  \mathbb{R}$ with $(\pr_i)_{\ast}(\nu)(A_i)\geq \kappa /k$, we have
  \begin{align*}
   \nu (A_1 \times A_2 \times \cdots \times A_k)=\nu \Big(
   \bigcap_{i=1}^k (\pr_i)^{-1}(A_i)\Big)\geq m- \kappa,
   \end{align*}which leads to
  \begin{align*}
   \diam (\nu,m-\kappa)\leq \diam (A_1 \times A_2 \times \cdots \times A_k)
   \leq k^{1/p} \max_{1\leq i \leq k} \diam A_i.
   \end{align*}We therefore get
  \begin{align*}
   \diam (\nu,m-\kappa)\leq k^{1/p} \max_{1\leq i\leq k} \diam \Big((\pr_i)_{\ast}(\nu),m-\frac{\kappa}{k}\Big).
   \end{align*}Combining this with Lemmas \ref{l2.1} and \ref{l2.3}, we obtain
  \begin{align*}
   \diam (\nu,m-\kappa)\leq k^{1/p} \max_{1\leq i\leq k}\sep \Big(
   (\pr_i)_{\ast}(\nu); \frac{\kappa}{2k}, \frac{\kappa}{2k}\Big)
   \leq k^{1/p}\sep \Big(\nu;\frac{\kappa}{2k}, \frac{\kappa}{2k}\Big).
   \end{align*}This completes the proof.
  \end{proof}
 \end{lem}

 \begin{lem}\label{pll1}Let $a,b$ be two real numbers with $a<b$. Then, a sequence
  $\{ X_n \}_{n=1}^{\infty}$ of mm-spaces is a L\'{e}vy family if and
  only if
  \begin{align}\label{spll1}
   \lim_{n\to \infty}\obs_{[a,b]}(X_n;-\kappa)=0 \text{ for any }\kappa >0.
   \end{align}
  \begin{proof}The necessity is obvious. We shall prove the
   converse. Suppose that the sequence $\{ X_n\}_{n=1}^{\infty}$ with the property (\ref{spll1}) is not
   a L\'{e}vy family. Then, by Corollary \ref{c2.1.1}, there exists $\kappa
   >0$ and Borel subsets $A_n, B_n\subseteq X_n$ such that $\mu_{X_n}(A_n)\geq \kappa$, $\mu_{X_n}(B_n)\geq \kappa$,
   and $\limsup_{n\to \infty}\dist_{X_n}(A_n,B_n)>0$. Define a
   function $f_n:X_n \to \mathbb{R}$ by
   $f_n(x):=\max\{\dist_{X_n}(x,A_n)+a, b\}$. Since $\mu_{X_n}(B_n)\geq
   \kappa$ and $\limsup_{n\to \infty}\dist_{X_n}(A_n,B_n)>0$, we have
   \begin{align*}
    \limsup_{n\to \infty}\diam ((f_n)_{\ast}(\mu_{X_n}),m_{X_n}-\kappa')>0
    \end{align*}for any $0<\kappa'<\kappa$. Since each $f_n$ is a
   $1$-Lipschitz function, this contradicts the assumption (\ref{spll1}). This completes the proof.
   \end{proof}
  \end{lem}
     \section{Proof of the main theorem}

     To prove the main theorem, we extract from Gournay's paper \cite{gournay}
     and Tsukamoto's paper
     \cite{tsukamoto} their arguments.

     For $k\in \mathbb{N}$, we identify
     $\mathbb{R}^k$ with the subset $\{ (x_1, x_2, \cdots, x_k, 0, 0 , \cdots )\in
       \mathbb{R}^{\infty} \mid x_i \in \mathbb{R} \text{ for all }i\}$
     of $\mathbb{R}^{\infty}$. Given $k\in \mathbb{N}\cup \{ \infty\}$,
     let $\mathfrak{S}_k$ be the $k$-th symmetric group. We consider the group $G_{k}:=
     \{\pm 1\}^{k} \rtimes \mathfrak{S}_{k}$. The multiplication in
     $G_{k}$ is given by
     \begin{align*}
      (( \varepsilon_n)_{n=1}^{k}, \sigma )\cdot ((
      \varepsilon'_n    )_{n=1}^{k}, \sigma'):= ((\varepsilon_n
      \varepsilon'_{\sigma^{-1}(n)})_{n=1}^{k}, \sigma \sigma').
      \end{align*}The group $G_k$ acts on the space $\mathbb{R}^{k}$ by
      \begin{align*}
       (( \varepsilon_n)_{n=1}^{k}, \sigma)\cdot (
       x_n)_{n=1}^{k}:= ( \varepsilon_n x_{\sigma^{-1}(n)})_{n=1}^{k}.
       \end{align*}Note that this action preserves the $k$-dimensional $\ell^p$-ball
       $B_{\ell^p}^{k}\subseteq B_{\ell^p}^{\infty}$ and the $\ell^q$-distance function
       $\dist_{\ell^q}$. Define a subset $\varLambda_{k}\subseteq
       B_{\ell^p}^k$ by
       \begin{align*}
        \varLambda_{k}:= \{ x  \in
        B_{\ell^p}^{k} \mid x_{i-1} \geq  x_i \geq 0 \text{ for all }i\}.
        \end{align*}
        Given an arbitrary $\varepsilon >0$, we put $ k(\varepsilon):=
        \lceil ( 2/\varepsilon)^{pq/(q-p)} \rceil -1$, where $\lceil (
        2/\varepsilon)^{pq/(q-p)} \rceil$ denotes the smallest integer
        which is not less than $( 2/\varepsilon)^{pq/(q-p)}$. For $k\geq
        k(\varepsilon)+1$, we define a continuous map $f_{k,\varepsilon}:
        \varLambda_{k }\to \mathbb{R}^{k(\varepsilon)}$ by
        \begin{align*}
         f_{k, \varepsilon} (x):=(x_1-x_{k(\varepsilon)+1}, x_2-x_{k(\varepsilon)+1},
         \cdots , x_{k(\varepsilon)}- x_{k(\varepsilon)+1},0,0 , \cdots).
         \end{align*}For any $x\in B_{\ell^p}^{k}$, taking $g\in
         G_{k}$ such that $gx\in \varLambda_{k}$, we define
         \begin{align*}
          F_{k,\varepsilon}(x):= g^{-1} f_{k,\varepsilon} (g x).
          \end{align*}This definition of the map
          $F_{k,\varepsilon}:B_{\ell^p}^{k} \to B_{\ell^p}^{k}$
          is well-defined (see \cite[Section 2]{tsukamoto} for details).
          Given $k\in \mathbb{N}$, we put $A_k:= \bigcup_{g\in
          G_{\infty}}g \mathbb{R}^k \subseteq \mathbb{R}^{\infty}$. 
       \begin{thm}[{cf.~\cite[Proposition 1.3]{gournay} and \cite[Section 2]{tsukamoto}}]\label{p3.1}The map
        $F_{k,\varepsilon}:B_{\ell^p}^{k}\to B_{\ell^p}^{k}$
        satisfies that $F_{k,\varepsilon}(B^{k}_{\ell^p})\subseteq
        A_{k(\varepsilon)}$ and
        \begin{align}\label{s3.1}
         \dist_{\ell^q}(x,F_{k,\varepsilon}(x))\leq \frac{\varepsilon}{2}
         \end{align}for any $x \in B_{\ell^p}^{k}$.
        \end{thm}

       \begin{lem}\label{l3.1}The map
        $F_{k,\varepsilon}:(B^{k}_{\ell^p},\dist_{\ell^q}) \to 
        (A_{k(\varepsilon)}, \dist_{\ell^q})$ is a $(1+k(\varepsilon)^{1/q})$-Lipschitz map.
        \begin{proof}By the definition of the map $F_{k,\varepsilon}$,
         it suffices to prove that the map
         $F:=F_{2k(\varepsilon)+2,\varepsilon}:(B^{2k(\varepsilon)+2}_{\ell^p},\dist_{\ell^q})
         \to (B^{2k(\varepsilon)+2}_{\ell^p},\dist_{\ell^q})$ is
         $(1+k(\varepsilon)^{1/q})$-Lipschitz. Recall that 
	 \begin{align*}
	  F(x)=(x_1- x_{k(\varepsilon)+1}, x_2-x_{k(\varepsilon)+1},
	  \cdots , x_{k(\varepsilon)}-x_{k(\varepsilon)+1},0,0,\cdots,0 )
	  \end{align*}for any $ x\in \varLambda_{2k(\varepsilon)+2}$. We
         hence get
      \begin{align*}
       \dist_{\ell^q}(F(x),F(y))\leq \dist_{\ell^q}(x,y)+
       k(\varepsilon)^{1/q}|x_{k(\varepsilon)+1}-y_{k(\varepsilon)+1}|\leq (1+k(\varepsilon)^{1/q})\dist_{\ell^q}(x,y)
       \end{align*}for any $x,y\in
         \varLambda_{2k(\varepsilon)+2}$. Since each $g\in G_{2k(\varepsilon)+2}$ preserves the distance function $\dist_{\ell^q}$, the map $F$ is $(1+k(\varepsilon)^{1/q})$-Lipschitz on
         each $g\varLambda_{2k(\varepsilon)+2}$. 

Let $x,y\in B_{\ell^p}^{2k(\varepsilon)+2}$ be arbitrary points. Observe
	 that there exist $t_0:=0\leq t_1 \leq t_2 \leq \cdots
	 \leq t_{i-1} \leq  1=:t_i$ and $g_1,g_2, \cdots , g_i\in
         G_{2k(\varepsilon)+2}$ such that $(1-t)x + ty \in g_j
         \varLambda_{2k(\varepsilon)+2}$ for any $t\in
         [t_{j-1},t_j]$. We therefore obtain
         \begin{align*}
          \dist_{\ell^q}(F(x),F(y))\leq \ &\sum_{j=1}^{i}
          \dist_{\ell^q}(F((1-t_{j-1})x+ t_{j-1}y), F((1-t_j)x+ t_j
          y))\\
          \leq \ & (1+k(\varepsilon)^{1/q}) \sum_{j=1}^{i} \dist_{\ell^q}( (1-t_{j-1})x+ t_{j-1}y
          , (1-t_j)x + t_j y)\\
          = \ &(1+k(\varepsilon)^{1/q})\dist_{\ell^q}(x,y).
          \end{align*}This completes the proof.
         \end{proof}
        \end{lem}

         The following lemma is a key to prove Theorem \ref{mth1}. 

         \begin{lem}\label{l3.2}Let $k\in \mathbb{N}$ and $\{\nu_{n,k} \}_{n=1}^{\infty}$ be a sequence
          of finite Borel measures on
          $(A_{k},\dist_{\ell^q})$ satisfying that
          \begin{align}\label{sl1}
           \lim_{n\to \infty}\sep (\nu_{n,k};\kappa_1,\kappa_2)=0
           \end{align}for any $\kappa_1,\kappa_2 >0$. Then, putting
          $m_n:=\nu_{n,k}(A_k)$, we have
          \begin{align}\label{sl2}
           \lim_{n\to \infty}\diam (\nu_{n,k},m_n-\kappa)=0
           \end{align}for any $\kappa >0$.
          \begin{proof}It suffices to prove (\ref{sl2}) by choosing a
           subsequence. We shall prove it by induction for $k$.

           For $k=0$, since $A_0= \{ (0,0, \cdots)\}$, we have $\diam
           (\nu_{n,0},m_n- \kappa)=0$.

           Assume that (\ref{sl2}) holds for any sequence $\{
           \nu_{n,k-1}\}_{n=1}^{\infty}$ of finite Borel measures on $(A_{k-1},\dist_{\ell^q})$
           having the property (\ref{sl1}). Let $\{
           \nu_{n,k}\}_{n=1}^{\infty}$ be any sequence of finite Borel
           measures on $(A_{k},\dist_{\ell^q})$ having the property
           (\ref{sl1}). Since $\lim_{n\to \infty}m_n=0$ implies (\ref{sl2}), we assume
           that $\inf_{n\in \mathbb{N}}m_n >0$. Putting
         \begin{align*}
          a_n:=\max \Big\{ \sep
          \Big(\nu_{n,k};\frac{m_n}{6},
          \frac{\kappa}{2}\Big),  \sep \Big( \nu_{n,k};
          \frac{m_n}{6}, \frac{m_n}{6}\Big) \Big\},
          \end{align*}we get $\lim_{n\to \infty}a_n=0$
         by the assumption (\ref{sl1}) and $\inf_{n\in \mathbb{N}}m_n >0$. Define subsets
         $B_{n,1}$ and $B_{n,2}$ of the set $A_{k}$ by
         $B_{n,1}:=(A_{k-1})_{a_n}\cap
         A_{k}$ and $B_{n,2}:=A_{k}\setminus
         B_{n,1}$, where $(A_{k-1})_{a_n}$ denotes the closed
           $a_n$-neighborhood of $A_{k-1}$. Since $A_{k}=B_{n,1}\cup B_{n,2}$, either
         the following $(1)$ or $(2)$ holds:

         (1) $\nu_{n,k}(B_{n,1})\geq m_n / 2$ for any
         sufficiently large $n\in \mathbb{N}$.

         (2) $\nu_{n,k}(B_{n,2})\geq m_n /2$ for infinitely
         many $n\in \mathbb{N}$.

         We first consider the case (2). We denote by $\mathcal{C}_n$
         the set of all connected components of the set $B_{n,2}$. 

         \begin{claim}\label{cl1}There exists $C_n \in \mathcal{C}_n$ such that
          $\nu_{n,k}(C_n)\geq m_n/6$.
          \begin{proof}If $\nu_{n,k}(C)< m_n/6$ for all
           $C\in \mathcal{C}_n$, then there exists $\mathcal{C}_n'
           \subseteq \mathcal{C}_n$ such that
           \begin{align*}
            \frac{m_n}{6}\leq \nu_{n,k}\Big( \bigcup_{C'\in
            \mathcal{C}_n'}C' \Big) < \frac{m_n}{3}
            \end{align*}because of $\nu_{n,k}(B_{n,2})\geq
           m_n/2$. Putting $\mathcal{C}_n'':= \mathcal{C}_n \setminus
           \mathcal{C}_n'$, we therefore obtain
           \begin{align*}
            2a_n \leq \dist_{\ell^q}\Big(   \bigcup_{C' \in \mathcal{C}_n'} C'   ,
            \bigcup_{C'' \in \mathcal{C}_n''} C''        \Big)\leq \sep
            \Big(\nu_{n,k};\frac{m_n}{6}, \frac{m_n}{6}\Big)<a_n,
            \end{align*}which is a contradiction. This completes the
           proof of the claim.
           \end{proof}
          \end{claim}

         \begin{claim}\label{cl2}Putting $D_n:=(C_n)_{\sep
          (\nu_{n,k};m_n/6, \kappa /2)}\cap
          A_{k}$, we have $\nu_{n,k}(D_n)\geq m_n-\kappa/2$.
          \begin{proof}Take any $\delta>0$. Supposing that
           $\nu_{n,k}((D_n)_{\delta} )<m_n-\kappa /2$, by Claim \ref{cl1}, we get 
           \begin{align*}
            \sep \Big(\nu_{n,k};\frac{m_n}{6}, \frac{\kappa}{2}\Big)<  \dist_{\ell^q}(C_n, A_{k}\setminus (D_n)_{\delta})\leq \sep \Big(\nu_{n,k};\frac{m_n}{6}, \frac{\kappa}{2}\Big),
            \end{align*}which is a contradiction. This proves that
           $\nu_{n,k}((D_n)_{\delta})\geq m_n-\kappa$
           for any $\delta >0$. Tending $\delta \to 0$, we obtain the claim.
           \end{proof}
          \end{claim}

         Observe that $D_n$ is isometrically embbeded into the $\ell^q$-space $(\mathbb{R}^{k},
         \dist_{\ell^q})$. Combining Lemma \ref{l2.4} and Claim \ref{cl2}, we therefore
         obtain
         \begin{align*}
          \diam (\nu_{n,k},m_n-\kappa)\leq \ & \diam
          (\nu_{n,k}|_{D_n},m_n-\kappa)\\
          \leq \ & \diam \Big(\nu_{n,k}|_{D_n},
          \nu_{n,k}(D_n) - 
          \frac{\kappa}{2} \Big) \\
          \leq \ & k^{1/q}\sep
          \Big(\nu_{n,k}|_{D_n};\frac{\kappa}{4k},
          \frac{\kappa}{4k} \Big)\\
          \leq \ & k^{1/q}\sep
          \Big(\nu_{n,k};\frac{\kappa}{4k},
          \frac{\kappa}{4k} \Big) \to 0 \text{ as }n\to \infty.
          \end{align*}This implies (\ref{sl1}).

         We next consider the case (1). Putting $b_n:= a_n+
         \sep (\nu_{n,k}; m_n / 2 , \kappa
         /2)$, as in the proof of Claim \ref{cl2}, we get
         \begin{align*}
          \nu_{n,k}((A_{k-1})_{b_n}
          \cap A_{k})
          =\nu_{n,k} ((B_{n,1})_{\sep
          (\nu_{n,k}; m_n /2, \kappa /2) }
          \cap A_{k})\geq m_n - \frac{\kappa}{2}.
          \end{align*}Note that there exists a Borel measurable map
         $f_{n}:(A_{k-1})_{b_n}\cap A_{k}\to A_{k-1}$ such that
         \begin{align}\label{sl3}
          \dist_{\ell^q}(x,f_{n}(x))=\min \{ \dist_{\ell^q} (x,y) \mid
          y \in A_{k-1}\}\leq b_n
          \end{align}for any $x\in (A_{k-1})_{b_n}\cap A_{k}$. Put $\nu_{n,k-1}:=
         (f_{n})_{\ast}(\nu_{n,k}|_{(A_{k-1})_{b_n}\cap
         A_{k}})$. An easy calculation proves that
           \begin{align*}
            \sep (\nu_{n,k-1};\kappa_1,\kappa_2)\leq \sep
            (\nu_{n,k};\kappa_1,\kappa_2) + 2b_{n}
            \end{align*}for any $\kappa_1,\kappa_2>0$. By this and the property (\ref{sl1}) for
           $\nu_{n,k}$, the measures $\nu_{n,k-1}$
         on $A_{k-1}$ satisfy that
           \begin{align*}
            \lim_{n\to \infty}\sep (\nu_{n,k-1};\kappa_1,\kappa_2) =0
            \end{align*}for any $\kappa_1,\kappa_2>0$. By the assumption of the induction, we therefore get
           \begin{align*}
            \lim_{n\to \infty}\diam \Big(\nu_{n,k-1},\nu_{n,k-1}(A_{k-1})-\frac{\kappa}{2}\Big)=0
            \end{align*}for any $\kappa >0$. By using (\ref{sl3}), we finally obtain
           \begin{align*}
            \diam (\nu_{n,k}, m_n-\kappa)\leq \ & \diam
            (\nu_{n,k-1},m_n-\kappa)+ 2b_n \\
            \leq \ & \diam (\nu_{n,k-1}, \nu_{n,k-1}(A_{k-1})-\kappa/2)
            + 2b_n \to 0 \text{ as }n\to \infty.
            \end{align*}This completes the proof of the lemma.
           \end{proof}
         \end{lem}
        \begin{proof}[Proof of Theorem \ref{mth1}]Lemma \ref{pll1} directly
         implies the sufficiency of
         Theorem \ref{mth1}. We shall prove the converse. Let $\{f_n:X_n \to
         (B^{\infty}_{\ell^p},\dist_{\ell^q})\}_{n=1}^{\infty}$ be any
         sequence of $1$-Lipschitz maps. Given an arbitrary $\varepsilon>0$, we
         shall prove that
         \begin{align*}
          \diam ((f_n)_{\ast}(\mu_{X_n}),m_{X_n}-\kappa)\leq 2\varepsilon
          \end{align*}for any $\kappa>0$ and any sufficiently large
         $n\in \mathbb{N}$. Put $k:=k(\varepsilon)$ and $\nu_{n, k}:=
         (F_{\infty, \varepsilon}
         \circ f_n)_{\ast}(\mu_{X_n})$. Since
         \begin{align*}
          \diam ((f_{n})_{\ast}(\mu_{X_n}),m_{X_n} -\kappa)\leq
          \diam(\nu_{n,k}, m_{X_n}-\kappa) +\varepsilon
          \end{align*}by (\ref{s3.1}), it suffices to prove that
         \begin{align}\label{mths1}
          \lim_{n\to \infty}\diam(\nu_{n,k}, m_{X_n}-\kappa)=0.
          \end{align}Since Lemma \ref{l2.1} together with Corollary \ref{c2.1.1}
         and Lemma \ref{l3.1} implies that
         \begin{align*}
          \sep (\nu_{n,k};\kappa_1,\kappa_2)\leq
          (1+k(\varepsilon)^{1/q})\sep(X_n;\kappa_1,\kappa_2)\to 0 \text{ as }n\to \infty
          \end{align*}for any $\kappa_1,\kappa_2>0$, by virtue of Lemma \ref{l3.2},
         we obtain (\ref{mths1}). This completes the proof.
        \end{proof}

         \section{Case of $1\leq q\leq p\leq +\infty$}

          For an mm-space $X$, we define the \emph{concentration function}
        $\alpha_X:(0,+\infty)\to \mathbb{R}$ as the supremum of
        $\mu_X(X\setminus A_{+r})$, where $A$ runs over all Borel subsets
        of $X$ with $\mu_X(A)\geq m_X/2$ and $A_{+r}$ is an open
        $r$-neighborhood of $A$. 
        \begin{lem}[{cf.~\cite[Corollary 2.6]{funagr}}]\label{l4.1}A sequence $\{X_n\}_{n=1}^{\infty}$ of mm-spaces is a
         L\'{e}vy family if and only if $\lim_{n\to
         \infty}\alpha_{X_n}(r)=0$ for any $r>0$. 
        \end{lem}
Let $p\geq 1$. We shall consider the $\ell^n_p$-sphere
        $\mathbb{S}^n_{\ell^p}:=\{ (x_i)_{i=1}^n \in \mathbb{R}^n \mid
        \sum_{i=1}^{\infty}|x_i|^p =1\}$. We denote by $\mu_{n,p}$ the
        cone measure and $\nu_{n,p}$ the
        surface measure on $\mathbb{S}^n_{\ell^p}$ normalized as
        $\mu_{n,p}(\mathbb{S}^n_{\ell^p})=\nu_{n,p}(\mathbb{S}^n_{\ell^p})=1$. In
        other words, for any Borel subset $A\subseteq
        \mathbb{S}^n_{\ell^p}$, we put
        \begin{align*}
         \mu_{n,p}(A):= \frac{1}{\mathcal{L}(B^n_{\ell^p})} \cdot \mathcal{L}(\{tx \mid x\in A \text{ and
         }0\leq t\leq 1\}),
         \end{align*}where $\mathcal{L}$ is the Lebesgue measure on $\mathbb{R}^n$.

         By the works of G. Schechtman
        and J. Zinn \cite[Theorems 3.1 and 4.1]{schzin} and R. Lata{\l}a and
        J. O. Wojtaszczyk \cite[Theorem 5.31]{latala}, we obtain
        \begin{align}
         \alpha_{(\mathbb{S}^{n}_{\ell^p},\dist_{\ell^2},\mu_{n,p})}(r)\leq
         C\exp (-cnr^{\min\{ 2,p \}}).
        \end{align}This inequality for $p\geq 2$ is also mentioned by
         A. Naor in \cite[Introduction]{naor} (see also
         \cite[Proposition 5.21]{latala}).
    \begin{lem}\label{l4.2}Let $1\leq q\leq p\leq +\infty$. Then, we have
     \begin{align*}
      \alpha_{(\mathbb{S}^{n}_{\ell^p},\dist_{\ell^q},\mu_{n,p})}(r)\leq
      C\exp(-cn^{1+(1/2-1/q)\min\{2,p\}}r^{\min \{2,p\}}) \text{
      if }q<2
      \end{align*}and 
     \begin{align*}
       \alpha_{(\mathbb{S}^{n}_{\ell^p},\dist_{\ell^q},\mu_{n,p})}(r)\leq
      C\exp(-cnr^{\min\{2,p\}}) \text{ if }q\geq 2.
      \end{align*}
     \begin{proof}If $q<2$, by $\dist_{\ell^q}(x,y)\leq
      n^{1/q-1/2}\dist_{\ell^2}(x,y)$, we then have
      \begin{align*}
       \alpha_{(\mathbb{S}^n_{\ell^p},\dist_{\ell^q},\mu_{n,p})}(r)\leq
       \alpha_{(\mathbb{S}^n_{\ell^p},\dist_{\ell^2},\mu_{n,p})}(n^{1/2-1/q}r)\leq  C\exp(-cn^{1+(1/2-1/q)\min\{2,p\}}r^{\min \{2,p\}}).
       \end{align*}If $q\geq 2$, by $\dist_{\ell^q}(x,y)\leq \dist_{\ell^2}(x,y)$,
      we then obtain
      \begin{align*}\alpha_{(\mathbb{S}^n_{\ell^p},\dist_{\ell^q},\mu_{n,p})}(r)\leq
       \alpha_{(\mathbb{S}^n_{\ell^p},\dist_{\ell^2},\mu_{n,p})}(r)\leq C\exp(-cnr^{\min\{2,p\}}).
       \end{align*}This completes the proof.
      \end{proof}
     \end{lem}

     \begin{cor}\label{c4.1}The sequences $\{ (\mathbb{S}^n_{\ell^p},
     \dist_{\ell^q}, \mu_{n,p})\}_{n=1}^{\infty}$ and $\{
     (\mathbb{S}^n_{\ell^p}, \dist_{\ell^q},
     \nu_{n,p})\}_{n=1}^{\infty}$ are both L\'{e}vy families for $1\leq q\leq
      p\leq +\infty$.
      \begin{proof}Since $1+(1/2-1/q)\min\{2,p\}>0$, by Lemmas \ref{l4.1}
       and \ref{l4.2},
       the sequence $\{ (\mathbb{S}^n_{\ell^p},
     \dist_{\ell^q}, \mu_{n,p})\}_{n=1}^{\infty}$ is a L\'{e}vy
       family. By virtue of \cite[Theorem 6]{naor}, the sequence $\{\mathbb{S}^n_{\ell^p}, \dist_{\ell^q},
     \nu_{n,p})  \}_{n=1}^{\infty}$ is also a L\'{e}vy family. This
       completes the proof.
       \end{proof}
      \end{cor}
 	\begin{prop}\label{p4.1}Let $1\leq q\leq p\leq +\infty$. Then, for any $\kappa$ with $0<\kappa <1/2$, we have
	 \begin{align*}
      \obs_{(B_{\ell^p}^{\infty},\dist_{\ell^q})}
	  ((\mathbb{S}^n_{\ell^p},\dist_{\ell^q},\mu);-\kappa) \geq 2,
	  \end{align*}where $\mu=\mu_{n,p}$ or $\mu=\nu_{n,p}$.
     \begin{proof}Let $A\subseteq \mathbb{S}^{n}_{\ell^p}$ be a Borel
	  subset such that $\mu(A)\geq
         1-\kappa$. Since $\mu(A)=
         \mu (-A) >1/2$, we have
         $\mu(A \cap (-A))>0$. Hence, there exists $x\in A$ such that
      $-x\in A$. Since $\diam A \geq \dist_{\ell^q}(x,-x)\geq
         \dist_{\ell^p}(x,-x) = 2$, we obtain $\diam
         (\mu,1-\kappa) =2$. Since the inclusion map from
      the space $(\mathbb{S}^n_{\ell^p},\dist_{\ell^q})$ to the space
      $(B^{\infty}_{\ell^p},\dist_{\ell^q})$ is $1$-Lipschitz, we obtain
      the conclusion. This completes the proof.
      \end{proof}
	 \end{prop}

     Combining Corollary \ref{c4.1} with Proposition \ref{p4.1}, we obtain an example of a
     L\'{e}vy family which does not satisfy (\ref{smth1}) in the case of $1\leq
     q\leq p \leq +\infty$.

\begin{ack}\upshape
        The author would like to express his thanks to
  Professor Takashi Shioya for his valuable suggestions and assistances
 during the preparation of this paper. The author thanks Professor Masaki
 Tsukamoto for telling his interesting work \cite{tsukamoto}. The author
 also thanks Professor A. Naor for his useful comments. Without them,
 this work would have never been completed.
 \end{ack}


\begin{thebibliography}{20}
	 \bibitem{funath}K. Funano, {\it Asymptotics of mm-spaces}, Doctoral thesis, in preparation.
          \bibitem{funano2}K. Funano, {\it Central and $L^p$-concentration of
             $1$-Lipschitz maps into $\mathbb{R}$-trees}, to appear in
                  J. Math. Soc. Japan.
          \bibitem{funagr}K. Funano, {\it Concentration of maps and
                  group action}, preprint, available online at ``http://front.math.ucdavis.edu/0807.3210'', 2008.
          \bibitem{funad}K. Funano, {\it Observable concentration of
		 mm-spaces into spaces with doubling measures},
		 Geom. Dedicata {\bf 127}, 49--56, 2007.
     \bibitem{funano1}K. Funano, {\it Observable concentration of
             mm-spaces into nonpositively curved manifolds}, preprint,
             available online at
             ``http://front.math.ucdavis.edu/0701.5535'', 2007.
             \bibitem{giopes1}T. Giordano and V. Pestov, {\it Some
                     extremely amenable groups}, C. R. Acad. Sci. Paris,
                     S\'{e}r. I {\bf 334}, No. 4, 273--278, 2002.
             \bibitem{giopes2}T. Giordano and V. Pestov, {\it Some
                     extremely amenable groups related to operator
                     algebras and ergodic theory},
                     J. Inst. Math. Jussieu {\bf 6}, no. 2, 279--315, 2007.
             \bibitem{gla}S. Glasner, {\it On minimal actions of Polish
                     groups}, Top. Appl. {\bf 85}, 119--125, 1998.
            \bibitem{gournay}A. Gournay, {\it Width of $\ell^p$-balls}, preprint, available online at ``http://front.math.ucdavis.edu/0711.3081'', 2007.
     \bibitem{milgro}M. Gromov and V. D. Milman, {\it A topological application of the isoperimetric inequality},
			 Amer. J. Math. {\bf 105}, no. 4, 843--854, 1983.
             \bibitem{gromovcat}M. Gromov, {\it ${\rm CAT}(\kappa)$-spaces:
		 construction and concentration}, (Russian summary)
		 Zap. Nauchn. Sem. S.-Peterburg. Otdel. Mat. Inst. Steklov. (POMI)
                     280, Geom. i Topol. {\bf 7}, 100--140, 299--300, 2001; translation in
J. Math. Sci. (N. Y.) {\bf 119}, no. 2, 178--200, 2004. 
	 \bibitem{gromov2}M. Gromov, {\it Isoperimetry of waists and
			 concentration of maps}, Geom. Funct. Anal., {\bf 13}, no. 1,
			 178--215, 2003.
	 \bibitem{gromov}M. Gromov, {\it Metric structures for Riemannian and non-Riemannian spaces}, Based on the 1981 French
			 original, With appendices by M. Katz, P. Pansu and S. Semmes. Translated from the French by Sean Michael
			 Bates. Progress in Mathematics,
		 {\bf 152}. Birkh\"{a}user Boston, Inc., Boston, MA, 1999.
      \bibitem{grotop}M. Gromov, {\it Topological invariants of
              dynamical systems and spaces of holomorphic maps},
              I. Math. Phys. Anal. Geom. {\bf 2}, no. 4, 323--415, 1999.
             \bibitem{latala}R. Lata{\l}a and J. O. Wojtaszczyk, {\it On the infimum convolution
              inequality}, preprint, available online at
              ``http://arxiv.org/abs/0801.4036'', 2008.
             \bibitem{ledoux}M. Ledoux, {\it The concentration of
                     measure phenomenon}, Mathematical Surveys and
                     Monographs, {\bf 89}. American
			 Mathematical Society, Providence, RI, 2001.
       \bibitem{mil1}V. D. Milman, {\it A certain property of functions
		 defined on infinite-dimensional manifolds}, (Russian)
                        Dokl. Akad. Nauk SSSR {\bf 200}, 781--784, 1971.
         	\bibitem{mil2}V. D. Milman, {\it A new proof of A. Dvoretzky's
		 theorem on cross-sections of convex bodies}, (Russian)
		 Funkcional. Anal. i Prilo\v zen. {\bf 5}, no. 4, 28--37,
		 1971.
           \bibitem{mil3}V. D. Milman, {\it Asymptotic properties of
		 functions of several variables that are defined on
		 homogeneous spaces}, Soviet Math. Dokl. {\bf 12},
		 1277--1281, 1971; translated from Dokl. Akad. Nauk SSSR {\bf 199}
		 , 1247--1250, 1971(Russian).
            \bibitem{mil4}V. D. Milman, {\it Diameter of a minimal invariant
             subset of equivariant subset of equivariant Lipschitz
             actions on compact subsets of $\mathbb{R}^k$}, Grometric
             Aspects of Functional Analysis, Israel Seminar, 1985--1986. Lecture Notes in Math. {\bf 1267}, 13--20,
                    Springer, 1987.
        	\bibitem{mil5}V. D. Milman, {\it The heritage of P. L\'{e}vy in
		 geometrical functional analysis}, Colloque Paul L\'{e}vy
		 sur les Processus Stochastiques (Palaiseau, 1987).  Ast
		 \'{e}risque  No. {\bf 157}-{\bf 158}, 273--301, 1988.
        \bibitem{milsch}V. D. Milman, G. Schechtman, {\it Asymptotic
		 theory of finite-dimensional normed spaces. With an
		 appendix by M. Gromov}, Lecture Notes in Mathematics,
		 {\bf 1200}. Springer-Verlag, Berlin, 1986.
        \bibitem{naor}A. Naor, {\it The surface measure and cone measure
                on the sphere of $\ell^n_p$},
                Trans. Amer. Math. Soc. {\bf 359}, 1045--1079, 2007.
         \bibitem{pestov2}V. Pestov, {\it Dynamics of infinite-dimensional
		 groups. The Ramsey-Dvoretzky-Milman phenomenon,}
		 Revised edition of  Dynamics of infinite-dimensional
		 groups and Ramsey-type phenomena
		 [Inst. Mat. Pura. Apl. (IMPA), Rio de Janeiro, 2005;
		 MR2164572].University Lecture Series, {\bf 40}. American
		 Mathematical Society, Providence, RI, 2006.
         \bibitem{pestov1}V. Pestov, {\it Ramsey-Milman phenomenon,
                 Urysohn metric spaces, and extremely amenable
                 groups}. - Israel Journal of Mathematics {\bf 127},
                 317--358, 2002. {\it Corrigendum}, ibid., {\bf 145},
                 375--379, 2005. 
        \bibitem{pestov3}V. Pestov, {\it The isometry groups of the
                Urysohn metric space as a L\'{e}vy group}, Topology
                Appl. {\bf 154}, no. 10, 2173--2184, 2007.
        \bibitem{schzin}G. Schechtman and J. Zinn, {\it Concentration on
                the $\ell^n_p$ Ball}, Geometric aspects of functional
                analysis, 245--256, Lecture Notes in Math., {\bf 1745},
                Springer, Berlin, 2000.
        \bibitem{tsukamoto}M. Tsukamoto, {\it Macroscopic dimension of
                the $\ell^p$-ball with respect to the $\ell^q$-norm}, to
                appear in J. Math. Kyoto Univ. 
	\end{thebibliography}
	\end{document}